\documentclass[11pt]{amsart}

\usepackage{amsxtra,amssymb,amsmath,amscd,url,listings,stmaryrd}
\usepackage[alphabetic,initials]{amsrefs}
\usepackage{scalerel,stackengine}
\stackMath
\usepackage[utf8]{inputenc}
\usepackage{eucal}
\usepackage{fullpage}
\usepackage{scrtime}
\usepackage[colorlinks]{hyperref}
\usepackage{xcolor}
\usepackage{datetime}
\usepackage{graphicx}

\shortdate

\renewcommand{\leq}{\leqslant}
\renewcommand{\geq}{\geqslant}
\newcommand\widecheck[1]{%
\savestack{\tmpbox}{\stretchto{%
  \scaleto{%
    \scalerel*[\widthof{\ensuremath{#1}}]{\kern-.4pt\bigwedge\kern-.4pt}%
    {\rule[-\textheight/2]{1ex}{\textheight}}%WIDTH-LIMITED BIG WEDGE
  }{\textheight}% 
}{0.5ex}}%
\stackon[2pt]{#1}{\scalebox{-1}{\tmpbox}}%
}

\numberwithin{equation}{section}

\def\stacksum#1#2{{\stackrel{{\scriptstyle #1}}
{{\scriptstyle #2}}}}

\def\map#1#2#3#4{\begin{matrix}#1&\mapsto &#2
\\#3 &\mapsto &#4
\end{matrix}}

\newcommand{\FT}{\mathrm{FT}}

\newcommand{\Cc}{\mathbf{C}}

\newcommand{\Aa}{\mathbf{A}}
\newcommand{\Gm}{\mathbb{G}_{m}}

\newcommand{\Zz}{\mathbf{Z}}

\newcommand{\Rr}{\mathbf{R}}

\newcommand{\Qq}{\mathbf{Q}}

\newcommand{\Fq}{{\mathbf{F}_q}}

\newcommand{\Fqt}{{\mathbf{F}^\times_q}}
\newcommand{\Fqqt}{{\mathbf{F}^\times_{q_0}}}

\newcommand{\Ff}{\mathbf{F}}

\newcommand{\mcW}{\mathcal{W}}

\newcommand{\HYPK}{\mathcal{K}\ell}
\newcommand{\KL}{\mathcal{K}\ell}

\newcommand{\mods}[1]{\,(\mathrm{mod}\,{#1})}

\newcommand{\what}{\widehat}

\newcommand{\ra}{\rightarrow}

\DeclareMathOperator{\Kl}{\mathrm{Kl}}

\DeclareMathOperator{\swan}{Swan}

\newcommand{\eps}{\varepsilon}
\renewcommand{\rho}{\varrho}

\DeclareMathOperator{\GL}{GL}

\DeclareMathSymbol{\gena}{\mathord}{letters}{"3C}
\DeclareMathSymbol{\genb}{\mathord}{letters}{"3E}

\theoremstyle{plain}
\newtheorem{theorem}{Theorem}[section]
\newtheorem*{theorem*}{Theorem}
\newtheorem{lemma}[theorem]{Lemma}

\newtheorem{corollary}[theorem]{Corollary}

\newtheorem{proposition}[theorem]{Proposition}

\newtheorem{definition}[theorem]{Definition}

\newtheorem*{notation}{Notation}

\theoremstyle{remark}

\theoremstyle{definition}

\newtheorem{remark}[theorem]{Remark}
\newtheorem*{remark*}{Remark}

\newcommand{\mcL}{\mathcal{L}}

\newcommand{\mcC}{\mathcal{C}}

\newcommand{\mcF}{\mathcal{F}}
\newcommand{\mcZ}{\mathcal{Z}}

\newcommand{\mcK}{\mathcal{K}}

\newcommand{\vphi}{\varphi}

\renewcommand{\geq}{\geqslant}
\renewcommand{\leq}{\leqslant}
\renewcommand{\Re}{\mathfrak{Re}\,}

\newcommand{\ov}[1]{\overline{#1}}

\newcommand\sumsum{\mathop{\sum\sum}\limits}

\newcommand{\sumstar}{\sideset{}{^\star}\sum}

\begin{document}
\title{On algebraic twists with composite moduli, II}

\author{Yongxiao Lin}
\address{Data Science Institute, Shandong University, Jinan 250100, China; State Key Laboratory of Cryptography and Digital Economy Security, Shandong University, Jinan 250100, China}
\email{yongxiao.lin@sdu.edu.cn}

\author{Philippe Michel}
\address{EPFL/MATH/TAN, Station 8, CH-1015 Lausanne, Switzerland }
\email{philippe.michel@epfl.ch}

%\date{\today,\ \thistime} 

\begin{abstract} 
  We study bounds for correlation sums of automorphic coefficients on $\GL_{3,\Qq}$ with trace functions of composite moduli. This is a sequel of  \cite{KLMS} and \cite{Lin-Michel}.
\end{abstract}

\thanks{Y. \ L. was partially supported by the National Key R\&D Program of China (No. 2021YFA1000700). Ph.\ M. was partially supported by the SNF grant 197045 and the SNF-ANR ``ETIENE'' grant 10003145.  \today\ \currenttime}

\maketitle
%\setcounter{tocdepth}{1}
%\tableofcontents

\section{Introduction}

In this paper we pursue our investigations on algebraic twists of automorphic $L$-function: given some automorphic $L$-function of degree $d\geq 1$
$$L(\pi,s)=\sum_{n\geq 1}\frac{\lambda_\pi(n)}{n^s}=\prod_p L(\pi_p,s)$$
(normalized so that $\Re s=1/2$ is the critical line), an integer $q$ and an ``algebraic exponential sum''
$$K:\Zz/q\Zz\to \Cc$$
modulo $q$ and $V$ a smooth compactly supported function, our aim is to obtain non-trivial bounds for the correlation sum 
\begin{equation}\label{nontrivialgeneric}
\mcC_V(X;\pi,K):=\sum_{n\geq 1}\lambda_\pi(n)K(n)V\left(\frac{n}X\right)\stackrel{?}{\ll}_{K,\pi,V}X^{1-\delta}
\end{equation}
for some fixed $\delta>0$ as $q,X\ra\infty$ in suitable ranges. Two  ranges are of peculiar interest  as they appear as natural barriers in numerous classical problems in analytic number theory:
\begin{itemize}
    \item[--] The {\em convexity range}
    $$X_c=q^{d/2}.$$
    Indeed for $K=\chi$ a primitive Dirichlet character \eqref{nontrivialgeneric} is essentially equivalent to the subconvexity problem for the twisted $L$-function $L(\pi\times\chi,1/2)$.
    \item[--] The {\em arithmetic progression} (a.p.) range
    $$X_{ap}=q^{\frac{d-1}2}.$$
    Indeed for $K=\Kl_d$ an hyper-Kloosterman sum in $d$ variables, \eqref{nontrivialgeneric} is directly related to the distribution of $\lambda_{\pi}(n)$ in large arithmetic progressions, i.e., the problem of evaluating the sum
    $$\sum_\stacksum{n\leq N}{n\equiv a\mods q}\lambda_{\pi}(n)$$
    for $N$ as small as possible relative to $q$; here a natural barrier is $$N_{ap}=q^{\frac{d+1}{2}}={q^d}/{X_{ap}}.$$
    Obtaining \eqref{nontrivialgeneric} within the a.p. range is usually quite hard in general and one often needs extra assumptions on the multiplicative function $n\mapsto \lambda_\pi(n)$ (such as being a multiplicative convolution see \cite{FI}). An intermediate objective would be to obtain  \eqref{nontrivialgeneric} for $X$ as close as possible to $X_{ap}$.
\end{itemize}

In \cites{FKM1,KLMS,KLM, LMS} we obtained \eqref{nontrivialgeneric} for $\GL_2,\GL_3,\GL_4$ and $\GL_2\times\GL_3$ $L$-functions for $X$ between the convexity and arithmetic progression ranges, when the modulus is prime and for quite general algebraic exponential sums $K$ (then called trace functions). In \cite{Lin-Michel}, we extended the results of \cites{FKM1,LMS} (covering the $\GL_2$ and $\GL_2\times\GL_3$ cases) to composite moduli of the shape $q=q_0q_1$ (with $q_0$ a prime and $q_1\geq 2$ an integer coprime with $q_0$) and improved the outcome over the  prime modulus case when $q_0$ is in specific ranges.

In this note, we complete \cite{Lin-Michel} by revisiting  \cite{KLMS} (the $\GL_3$-case) for composite moduli.

To state our main result we need to introduce some notation.

\begin{itemize}
    \item We denote by $\varphi$ a Hecke--Maass cuspform for $\GL_{3,\Qq}$ of level $1$ with Fourier coefficients $(\lambda_\varphi(n,r))_{n\geq 1,r\not=0}$ which is fixed.
    \item Let $q$ be a modulus of the shape $q=q_0q_1$ with $q_0$ a prime and $q_1\geq 2$  an integer coprime with $q_0$. 

We give ourselves an algebraic exponential sum $K:\Zz/q\Zz\mapsto \Cc$ of the shape $K=K_0.K_1$ for two functions $$K_0:\Zz/q_0\Zz,\ K_1:\Zz/q_1\Zz,$$ whose product is defined by precomposition with the Chinese Reminder Theorem isomorphism:
$$K:n\mapsto K(n)=K_0(n\mods {q_0})K_1(n\mods {q_1}).$$
\begin{itemize}
   \item[$\bullet$]  We assume that $K_0$ is the trace function of a middle extension sheaf on $\Aa^1_{\Fq}$, say $\mcF$, which is pure of weight $0$, geometrically irreducible and non-trivial. We denote by $C_0:=C(\mcF)$ its complexity (cf. \cite{FKM1}*{Def. \,1.13}) and assume that $C_0$ remains absolutely bounded as $q\ra\infty$.

\item[$\bullet$] For the function $K_1$, we will only require that its normalized Fourier transform $\what K_1:\Zz/q_1\Zz\mapsto \Cc$ defined by
\begin{equation}\label{FourierTransform-K}
\what K_1(n):=\frac{1}{q_1^{1/2}}\sum_{x\in\Zz/q_1\Zz}K_1(x)e\left(\frac{nx}{q_1}\right)
\end{equation}
has its supnorm absolutely bounded as $q\ra\infty$.
\end{itemize}

\item Finally for $Z\geq 1$ a real number, let $V\in \mcC^\infty_c(\Rr)$ be a smooth function, compactly supported in the interval $[1,2]$ which satisfy for each $j\geq 0$,
\begin{equation}\label{bound-of-V}
    V^{(j)}(x)\ll_j Z^j.
\end{equation}
\end{itemize}

Our main (mild) assumption regarding $K_0$ is that $\mcF$ is what we call {\em good}:
\begin{definition}
	\label{gooddef}
	For $q\geq 2$ a prime. A middle extension sheaf $\mcF$ on $\Aa^1_{\Fq}$, pure of weight $0$,  geometrically irreducible and non trivial, is good if the local monodromy of $\mcF$ at infinity has no indecomposable summand with slope equal to $1$. 
\end{definition}
\begin{remark*} Such a sheaf $\mcF$ is automatically Fourier: that was the single  assumption made on the trace function $K$ in \cite{KLMS}.
\end{remark*}
\begin{theorem}\label{mainthm2} 
Let $\varphi$, $q_0,q_1,q$, $K_0,K_1,K$, $Z,V$ be as above.

We assume that $K_0$ is the trace function of an $\ell$-adic sheaf $\mcF$ which is good in the sense of Definition \ref{gooddef}.

Let $X\geq 1$ be such that $X\geq Z^4q^2q_0^{-2}$; for any $\eps>0$ and any integer $r\geq 1$, we have
\begin{multline}
	\label{d3x2boundcomposite}
	\mcC_V(X;\vphi,K)=\sum_{n\geq 1}\lambda_\varphi(n,r)K(n)V\left(\frac{n}{X}\right)\\
    \ll_{\eps,\varphi,C_0} 
      X^{o(1)} Z\|\widehat{K_1}\|_{\infty}r^{1/2}\Bigg(X^{1/4}q^{1/2}{q_0^{1/2}}
      +X^{3/4}{q_0}^{3/4}+\frac{X^{3/4}q^{1/2}}{q_0^{1/2}}\Bigg).
\end{multline}
\end{theorem}

 \begin{remark} 
 \par (a) To compare with \cite{KLMS}, we consider the case  $$q_0\asymp q^{2/5}$$ and $r=Z=1$; we obtain that for $X\geq q^{6/5}$
\begin{equation}\label{GL3-twist}
\mcC_V(X;\vphi,K)\ll_{\varphi,\|\widehat{K_1}\|_{\infty},C_0}X^{o(1)}\bigl(X^{1/4}q^{7/10}+X^{3/4}q^{3/10}\bigr).
\end{equation}
For $X=q^{3/2}$ (the convexity range) we obtain
$$\mcC_V(q^{3/2};\vphi,K)\ll_{\varphi,\|\widehat{K_1}\|_{\infty},C_0}q^{3/2-3/40+o(1)}.$$
This is better than the bound from \cite{KLMS}
$$\mcC_V(q^{3/2};\vphi,K)\ll_{\varphi,\|\widehat{K_1}\|_{\infty},C_0}q^{3/2-1/36+o(1)}$$
(for $q$ a prime).
Moreover \eqref{GL3-twist} is non-trivial {\rm(}i.e., \eqref{nontrivialgeneric} holds{\rm)} as soon as
\begin{equation}\label{range-GL3}
X\geq q^{3/2-3/10+\eta}
\end{equation}
for some $\eta>0$ (notice that $q^{3/2-3/10}=q^2/q_0^2$); while this does not yet reach the a.p. range, this is better than the range $X\geq q^{3/2-1/6+\delta}$ obtained in \cite{KLMS} for the prime modulus case; moreover this matches the range $X\geq |t|^{3/2-3/10+\delta}$ in the $t$-aspect (i.e., $K(n)\leadsto n^{it}$)
proven by Aggarwal \cite{Agg21} (see \cite{Huang21} for further extensions).

\par (b) Improving \eqref{range-GL3} to $$X\geq q^{1-\eta},\ \eta>0$$ for $\vphi$ cuspidal and some modulus $q$ remains a very interesting unsolved problem. When $\vphi$ is not cuspidal the a.p. range has been passed for $\lambda_\varphi(n,1)=d_3(n)=1\star 1\star 1(n)$ the ternary divisor function by Friedlander--Iwaniec in their groundbreaking work \cite{FI} and for $\lambda_\varphi(n,1)=1\star \lambda_f(n)$ with $f$ a $GL_2$-cusp form by Kowalski, Sawin and the second named author \cite{KMS} when $q$ is a prime. This was for $K=\Kl_k$ $k\geq 2$, hyper-Kloosterman sums; by now the result \cite{FKMS} allows to handle much more general classes of trace functions.
\end{remark}

In \eqref{GL3-twist} taking $K=\chi$ to be the primitive Dirichlet character of modulus $q$ (the conditions of Definition \ref{gooddef} are satisfied since a Kummer sheaf is tamely ramified at $\infty$) and applying the approximate functional equation for the twisted $L$-function $$L(\varphi\times \chi,s)=\sum_{n\geq 1}\frac{\lambda_{\varphi}(n,1)\chi(n)}{n^s},$$
we readily obtain that
\begin{corollary}
Assume that $q$ has a prime factor $q_0$ satisfying $$q_0= q^{2/5+o(1)}.$$
Then we have 
\begin{equation*}
    L(\varphi\times \chi,1/2)\ll_{\varphi} q^{3/4-3/40+o(1)}.
\end{equation*}
\end{corollary}
This improves the subconvexity exponent obtained in \cites{Munshi15b, Hol-Nel} in the case of the prime modulus $q$ and also the exponent for composite moduli obtained in \cite{Munshi15}.
\begin{remark} 
It is relatively straightforward, using the same method, to incorporate the $t$-component and obtain the bound
\begin{equation*}
    L(\varphi\times \chi,1/2+it)\ll_{\varphi} \big(q(|t|+2)\big)^{3/4-3/40+o(1)}.
    \end{equation*}
   This removes the restriction that $q<(|t|+2)^{8/7}$ in the work \cite{KKL} of K\i ral, Kuan, and Lesesvre, at the cost of requiring $q$ to have a divisor of appropriate size.
\end{remark}
\smallskip

Let $\lambda_{1\boxplus{\varphi}}(n)=1\star\lambda_\vphi(\bullet,1)(n)=\sum_{m|n}\lambda_{{\varphi}}(m,1)$. Applying Theorem \ref{mainthm2} to the trace function
$$K(n)=\Kl_4(an;q)$$
(the conditions of Definition \ref{gooddef} are satisfied since the hyper-Kloosterman sheaf $\KL_4$ is Fourier and its only slope at $\infty$ is $1/4$)
we have 
\begin{corollary}\label{mainarith} 
Using the notation from the
theorem. Assume that $q$ has a prime factor $q_0$ satisfying $$q_0= q^{2/5+o(1)}.$$
Given any $\eta>0$. For any $(a,q)=1$ and for $q$ satisfying
$$q\leq X^{25/61-\eta}=X^{2/5+3/305-\eta},$$
we have
$$\mathop{\sum_{n\geq 1}}_{n\equiv a\bmod q}\lambda_{1\boxplus\varphi}(n)V\left(\frac{n}{X}\right)-\frac{1}{\varphi(q)}\sum_\stacksum{n\geq 1}{(n,q)=1}\lambda_{1\boxplus\varphi}(n)V\left(\frac{n}{X}\right)\ll_{\varphi,\eta, V}  (\frac{X}{q})^{1-\delta}$$
for some $\delta=\delta(\eta)>0$.
\end{corollary}
This is better than \cite{KLM}*{Thm.\,1.6} where (for a prime modulus $q$) the level of distribution $2/5+1/260$ was obtained; interestingly, the exponent $3/305$ matches exactly the exponent $3/305$ of \cite{Huang24}*{Thm. 3} concerning the error term of the Rankin--Selberg problem\footnote {which itself was inspired by \cite{KLM}; see \cite{Huang24}*{p.2}.}. See also \cite{Zhu} in the case $q=p^k$.

\subsection*{Principle of the proof} 
The proof of Theorem \ref{mainthm2} is based on (variants of) the $\delta$-symbol method pioneered by Munshi \cite{Munshi15} in the case $K=\chi$. 

Here we follow closely the approach of \cites{LMS} and start by separating $\lambda(r,n)$ and $K(n)$ in the correlation sum
$$\sum_{n=1}^{\infty}\lambda(r,n)K(n)V\left(\frac{n}{X}\right)$$
through a form of the $\delta$-symbol method. As in \cite{Lin-Michel} we take advantage of the factorisation $q=q_0q_1$ through a ``conductor-lowering" trick (modulo ${q_0}$) in the application of the $\delta$-symbol. This gives us the flexibility of detecting the equation $n=0$ in two steps: $n\equiv 0\bmod {q_0}$ and $n/{q_0}=0$. The firt is via orthogonality of additive chracter and the second via the $\delta$-symbol of Duke, Friedlander, and Iwaniec \cite{DFI1.5}, but now with a reduced size in the choice of the parameter $C$.

Much as in \cite{KLMS}, the contribution of the factor $K_1$ is absorbed very softly and requires only that the $\infty$-norm of its Fourier transform $\what K_1$ be absolutely bounded (is $q_1$ a prime this follows if $K_1$ is the trace function of a sheaf with absolutely bounded complexity). Handling the contribution of $K_0$ however is more delicate and requires to bound correlation sums of the shape
$$\sum_{x\in\Ff_{q_0}}Z(v)\ov{Z'}(v+\delta)$$ where $Z,Z'$ are two trace functions formed from $K_0$ via multiplicative convolution against some pull-back of a Kloosterman sum in one variable (see Prop. 
\ref{sqrootcancel1} and \S \ref{secsqroot}). These sums are very similar to the sums treated in \cite{LMS}*{Prop 4.5} (although a bit simpler) and we follow closely the proof of loc.cit.

 \subsection*{Acknowledgments}  The authors would like to thank Tengyou Zhu for an inquiry that led us to writing this note.
 
\section{Proof of Corollary \ref{mainarith}}

In this section, we show how to derive Corollary  \ref{mainarith} from  Theorem \ref{mainthm2}.

 Let $q=q_0q_1$. We apply the \emph{duality principle} \cite{LMS}*{Cor. 9.2} to the sum
$$\sum_{n\geq 1}\lambda_{1\boxplus\varphi}(n)K(n)V\left(\frac{n}{X}\right)$$
where
$$K(n)=q^{1/2}\delta_{n\equiv a\bmod  q}.$$
Setting $\widecheck X=q^4/X$, this gives (up to some negligible error terms) 
$$
\mathop{\sum_{n\geq 1}}_{n\equiv a\bmod q}\lambda_{1\boxplus\varphi}(n)V\left(\frac{n}{X}\right)-\frac{1}{\varphi(q)} \mathop{\sum_{n\geq 1}}_{(n,q)=1}\lambda_{1\boxplus\varphi}(n)V\left(\frac{n}{X}\right)=\frac{X}{q^{5/2}}\sum_{n\geq 1}\lambda_{1\boxplus\ov{\varphi}}(n)\mathrm{Kl}_4(an;q)\widecheck{V}\left(\frac{n}{\widecheck{X}}\right).$$
Inserting the definition 
$$\lambda_{1\boxplus\ov{\varphi}}(n)=\sum_{lm=n}\lambda_{\ov{\varphi}}(m,1)$$ and subdividing the variables $l, m$ into dyadic ranges, we are reduced to bounding
\begin{equation}
    \label{MNsums}
    \frac{X}{q^{5/2}}\sum_{l\geq 1}\sum_{m\geq 1}\lambda_{\ov{\varphi}}(m,1)\mathrm{Kl}_4(al m;q)V_1(\frac{l }{L})V_2(\frac{m}{M})
\end{equation}
for $O(\log^2 X)$ many real numbers $L,M\geq 1$ satisfying 
\begin{equation}
    \label{LMupperbound}
    LM\ll \frac{q^4}{X}.
\end{equation}
Here $V_1,V_2$ are smooth functions compactly supported in $[1,2[$.

From now on we assume that $q_0$ is such that $$q_0=q^{2/5+o(1)}.$$ 
\subsection{The case $M\geq q^{6/5-\delta}$}
Applying Theorem \ref{mainthm2}, the double sum in \eqref{MNsums} is bounded by
\begin{equation*}
\begin{split}
 \frac{XL}{q^{5/2}}\big|\sum_{m\geq 1}\lambda_{\ov{\varphi}}(m,1)\mathrm{Kl}_4(al m;q)V_2(\frac{m}{M})\big|
\ll& X^{o(1)} \frac{XL}{q^{5/2}}M^{3/4}q^{3/10}\\
\ll& X^{o(1)} \frac{X}{q}\bigl(\frac{L^5 q^{36}}{X^{15}}\bigr)^{20},
\end{split}
\end{equation*}
which is $\leq (X/q)^{1-\delta}$ as long as 
$$q\leq X^{5/12-\eta}L^{-5/36}$$
for some fixed $\eta>0$.

\subsection{The case $L\geq q^{1/2}$}
According to the calculations in \cite{Lin-Michel}*{\S 3.2}, we have
\begin{equation*}
    \frac{X}{q^{5/2}}\sum_{l\geq 1}\sum_{m\geq 1}\lambda_{\ov{\varphi}}(m,1)\mathrm{Kl}_4(al m;q)V_1(\frac{l }{L})V_2(\frac{m}{M})\ll (X/q)^{1-\delta}
\end{equation*}
as long as 
$$ L\geq q^{8/15+\eta}.$$

\subsection{The case $L\leq q^{1-\eta}$} 
According to \cite{Lin-Michel}*{\S 3.3}, we know
\begin{equation*}
    \frac{X}{q^{5/2}}\sum_{l\geq 1}\sum_{m\geq 1}\lambda_{\ov{\varphi}}(m,1)\mathrm{Kl}_4(al m;q)V_1(\frac{l }{L})V_2(\frac{m}{M})\ll (X/q)^{1-\delta}
\end{equation*}
as long as 
$$   q\leq X^{2/5-\eta}L^{1/5}.$$

Let $L_0=X^{3/61}$ be the solution of 
$$X^{5/12}L^{-5/36}=X^{2/5}L^{1/5}=X^{2/5+3/305}.$$
By combining the above three cases and following the analysis presented in \cite{Lin-Michel}*{\S 3.4}, we know that there exists some $\delta=\delta(\eta)>0$ such that 
$$\mathop{\sum_{n\geq 1}}_{n\equiv a\bmod q}\lambda_{1\boxplus\varphi}(n)V\left(\frac{n}{X}\right)-\frac{1}{\varphi(q)} \mathop{\sum_{n\geq 1}}_{(n,q)=1}\lambda_{1\boxplus\varphi}(n)V\left(\frac{n}{X}\right)\ll (X/q)^{1-\delta}$$
as long as 
$$X^{2/5-\eta}\leq q\leq X^{25/61-\eta}.$$ 

This completes the proof of Corollary \ref{mainarith}.

\section{Proof of Theorem \ref{mainthm2}}

\subsection{First transformations} 
We assume that $Z$ satisfies
$1\leq Z\leq q$ and  $V\in \mcC^\infty_c(\Rr)$  satisfies the bound \eqref{bound-of-V}.
We suppress the dependence on $\varphi$ and denote by $\lambda(r,n)$ 
 the $(r,n)$-th Fourier coefficient of $\varphi$. As in \cite{KLMS}, we write
\begin{equation}\label{Sdef}
\mcC_{V,r}(X;K)=\sum_{n=1}^{\infty}\lambda(r,n)K(n)V\left(\frac{n}{X}\right).
\end{equation}

We can write $\mcC_{V,r}(X;K)$ as 
\begin{equation}
     \mcC_{V,r}(X;K)=\sum_{n=1}^{\infty}\lambda(r,n)\sum_{m=1}^{\infty}K(m)\delta_{n=m}U\left(\frac{n}{X}\right)V\left(\frac{m}X\right).    \label{SVXfirst}
\end{equation}
Here $U$ is a smooth function supported in $(1/100,100)$ and satisfying $U(x)=1$ for $x\in [1,2]$ and $U^{(j)}(x)\ll_j 1$ for $j\geq 0$.

Let $C=(X/q_0)^{1/2}$. We apply a version of the delta method due to Duke--Friedlander--Iwaniec (\cite{DFI1.5}) 
\begin{eqnarray*}
\delta_{n=0}&=&\frac{1}{C}\sum_\stacksum{c\leq  2C}{(c,q)=1}\frac{1}{c{q_0}}\sumstar_{u(c{q_0})}e\left(n\frac{u}{c{q_0}}\right)h\left(\frac{c}{C},\frac{n}{C^2{q_0}}\right)\nonumber\\
&+&\frac{1}{C}\sum_\stacksum{c\leq  2C}{(c,q)=1}\frac{1}{c{q_0}}\sumstar_{a(c)}e\left(n\frac{a}{c}\right)h\left(\frac{c}{C},\frac{n}{C^2{q_0}}\right)\label{eqdelta}\\
&+&
\frac{1}{C}\sum_{c\leq  2C/q}\frac{1}{cq{q_0}}\sumstar_{a(cq{q_0})}e\left(n\frac{a}{cq{q_0}}\right)h\left(\frac{cq}{C},\frac{n}{C^2{q_0}}\right)+O_A\left(C^{-A}\right)\nonumber
\end{eqnarray*}
as presented in \cite{HB}*{Thm. 1} and \cite{LMS}*{(3.7)} to $\delta_{n=m}$ in \eqref{SVXfirst}, and we obtain
\begin{equation}\label{S'sumErr}
\mcC_{V,r}(X;K)=\mathrm{Main}+\mathrm{Err}_1+\mathrm{Err}_2+O_A\left(X^{-A}\right)
\end{equation}
where
\begin{multline}\label{S'sum}
       \mathrm{Main}=\frac{1}{Cq_0}\sum_\stacksum{c\leq  2C}{(c,q)=1}\frac{1}{c}\sumstar_{u(cq_0)}\sum_{n=1}^{\infty}\lambda(r,n)e\left(\frac{un}{cq_0}\right)U\left(\frac{n}{X}\right)\\
        \times\sum_{m=1}^{\infty}K(m)e\left(\frac{-um}{cq_0}\right)V\left(\frac{m}X\right)h\left(\frac{c}{C},\frac{n-m}{C^2q_0}\right),
\end{multline}
and
\begin{multline*}
        \mathrm{Err}_1=\frac{1}{Cq_0}\sum_\stacksum{c\leq  2C}{(c,q)=1}\frac{1}{c}\sumstar_{a(c)}\sum_{n=1}^{\infty}\lambda(r,n)e\left(\frac{an}{c}\right)U\left(\frac{n}{X}\right)\\
        \times\sum_{m=1}^{\infty}K(m)e\left(\frac{-am}{c}\right)V\left(\frac{m}X\right)h\left(\frac{c}{C},\frac{n-m}{C^2q_0}\right),
\end{multline*}
\begin{multline*}
        \mathrm{Err}_2=\frac{1}{Cq{q_0}}\sum_\stacksum{c\leq  2C/q}{(c,q)=1}\frac{1}{c}\sumstar_{a(cqq_0)}\sum_{n=1}^{\infty}\lambda(r,n)e\left(\frac{an}{cqq_0}\right)U\left(\frac{n}{X}\right)\\
        \times\sum_{m=1}^{\infty}K(m)e\left(\frac{-am}{cqq_0}\right)V\left(\frac{m}X\right)h\left(\frac{cq}{C},\frac{n-m}{C^2{q_0}}\right).
\end{multline*}

As in \cite{Lin-Michel}, we focus our analysis on the term $\mathrm{Main}$ in \eqref{S'sum} which is the hardest and is responsible for the final bound. The treatment for the terms $\mathrm{Err}_1,\mathrm{Err}_2$ are very similar to the one presented in \cite{LMS}*{\S 7.1}, and just as in \cite{LMS} their contribution turns out to be smaller as compared to that of $\mathrm{Main}$. As such we completely skip their treatments and the reader is referred to \cite{LMS}*{\S 7.1} for the relevant details.

\subsubsection{Bounding $\mathrm{Main}$} 

We apply the Poisson summation formula to the $m$-sum in \eqref{S'sum}, to get
\begin{multline}
\sum_{m\geq 1}(\cdots)=\frac{X}{cq}\sum_{m\in \mathbb{Z}}\sum_{\beta(q)}K(\beta)e\left(\frac{-u\beta\bar{c}}{q_0}\right)e\left(\frac{m\beta\bar{c}}{q}\right)\sum_{\beta(c)}e\left(\frac{-\beta u\ov{q_0}}{c}\right)e\left(\frac{m\beta\bar{q}}{c}\right)\widehat{\mathcal{V}}\left(n,\frac{mX}{cq}\right)\\
=\frac{X}{q^{1/2}}\sum_\stacksum{m\in \mathbb{Z}}{u\equiv \ov{q_1}m(c)}\widehat{K}(\bar{c}(m-q_1u))\widehat{\mathcal{V}}\left(n,\frac{mX}{cq}\right)
\end{multline}
where $\widehat{K}$ is the normalized Fourier transform \eqref{FourierTransform-K} and \begin{equation}\label{FourierT-of-V}
\widehat{\mathcal{V}}(n, x)=\int_{\Rr}V(y)h\left(\frac{c}{C},\frac{n-yX}{C^2{q_0}}\right)e(-xy)\rm{d}y
\end{equation}is the Fourier transform of  $V(y)h\left(\frac{c}{C},\frac{n-yX}{C^2{q_0}}\right)$. We also note that the weight function $\widehat{\mathcal{V}}\left(n,\frac{mX}{cq}\right)$ restricts the effective range of the dual $m$-sum to 
\begin{equation}\label{truncate-m}
|m|\leq X^{o(1)}Zcq/X.
\end{equation}
 By this we mean that for any $\eps>0$, the contribution of the terms satisfying $$|m|> X^{\eps}Zcq/X$$
is bounded by $O_{A,\eps}(X^{-A})$ for any $A\geq 1$.

We find that  \eqref{S'sum} can be rewritten as
\begin{multline}\label{S'sum-2}
       \mathrm{Main}=\frac{1}{Cq_0}\sum_\stacksum{c\leq  2C}{(c,q)=1}\frac{1}{c}\sumstar_{u(cq_0)}\sum_{n=1}^{\infty}\lambda(r,n)e\left(\frac{un}{cq_0}\right)U\left(\frac{n}{X}\right)\\
        \times \frac{X}{q^{1/2}}\sum_\stacksum{m\in \mathbb{Z}}{u\equiv \ov{q_1}m(c)}\widehat{K}(\bar{c}(m-q_1u))\widehat{\mathcal{V}}\left(n,\frac{mX}{cq}\right).
\end{multline}

We next apply the Voronoi summation of \cite{Miller-Schmid} (see \cite{LMS}*{Prop. 2.2}) to the $n$-sum in \eqref{S'sum-2}, which gives
\begin{multline}\label{eqafter1stvoronoi}
           \sum_{n=1}^{\infty}\lambda(r,n)e\left(\frac{un}{cq_0}\right)U\left(\frac{n}{X}\right)\widehat{\mathcal{V}}\left(n,\frac{mX}{cq}\right)\\=c{q_0}\sum_{\pm}\sum_{n_1|rcq_0}\sum_{n\geq 1}\frac{\lambda(n,n_1)}{nn_1}S(r\ov u,\pm n;\frac{rcq_0}{n_1})\mcW_{\pm}(\frac{m}{cq/X},\frac{n_1^2n}{c^3q^3_0r/X})  ,
        \end{multline}
where after inserting \eqref{FourierT-of-V},
\begin{equation*}\label{double-weight-function}
\begin{split}
\mcW_{\pm}(y,z)=&z\int_{\mathbb{R}}V(x)\left(\int_{0}^{\infty}U(\xi)h\bigl(\frac{c}{C},\frac{X(\xi-x)}{C^2{q_0}}\bigr)J_{\vphi,\pm}(z\xi){\mathrm{d}\xi}\right) \, e(-xy){\mathrm{d}x};
  \end{split}\end{equation*}
  see \cite{Qi}*{(3.34)}  for the definition of $J_{\vphi,\pm}(\xi)$
  and  \cite{LMS}*{(3.19)} for details.
% with $$W_x(\xi):=U(\xi)h\left(\frac{c}{C},\frac{X(\xi-x)}{C^2{q_0}}\right);$$

If $z\gg 1$, by using integration by parts to the inner integral over $\xi$, we know that
\begin{equation}\label{decay-of-W}
\mcW_{\pm}(y,z)\ll_{A} (1+|z|)^{-A}.
\end{equation}
For $z\ll 1$, from \cite{LMS}*{(2.9)} we know that $\mcW_{\pm}(y,z)\ll z^{2/3}$; in particular, we have
\begin{equation}\label{Jacquet-Shalika}
\mcW_{\pm}(y,z)\ll z^{1/2}.
\end{equation}

Substituting back, we obtain a sum of the form
\begin{multline}\label{eqaftervoronoi}
       \mathrm{Main}=\frac{X}{Cq^{1/2}}\sum_{\pm}\sum_\stacksum{c\leq  2C}{(c,q)=1}\sum_{n_1|rcq_0}\sum_{n\geq 1}\frac{\lambda(n,n_1)}{nn_1}\\
       \times\sumstar_{u(cq_0)} \sum_\stacksum{m\in \mathbb{Z}}{u\equiv \ov{q_1}m(c)}\widehat{K}(\bar{c}(m-q_1u))S(r\ov u,\pm n;\frac{rcq_0}{n_1})\mcW_{\pm}(\frac{m}{cq/X},\frac{n_1^2n}{c^3q^3_0r/X}).
\end{multline}
Notice that for $(q_0,q_1)=1$ we have the twisted multiplicativity
% $$    \widehat{K}(b;{q_0}{q_1})=\widehat{K_0}(\overline{{q_1}}b;{q_0}) \widehat{K_1}(\overline{{q_0}}b;{q_1}),$$
\begin{equation*}
\widehat{K}(\bar{c}(m-q_1u);{q_0}{q_1})=\widehat{K_0}(\bar{c}(m-q_1u)\ov{q_1};{q_0}) \widehat{K_1}(\bar{c}m\ov{q_0};{q_1}).
\end{equation*}

We further split the sum in \eqref{eqaftervoronoi} into two subsums according to $(n_1,q_0)=1$ or not, and write 
$$ \mathrm{Main}= \mathrm{Main}_{0}+ \mathrm{Err}_3,$$
where
\begin{multline*}
       \mathrm{Main}_0=\frac{X}{Cq^{1/2}}\sum_{\pm}\sum_\stacksum{c\leq  2C}{(c,q)=1}\sum_\stacksum{n_1|rc}{(n_1,q_0)=1}\sum_{n\geq 1}\frac{\lambda(n,n_1)}{nn_1}\sumstar_{u(cq_0)} \sum_\stacksum{m\in \mathbb{Z}}{u\equiv \ov{q_1}m(c)}\widehat{K_1}(\bar{c}m\ov{q_0};{q_1})\\
       \times \widehat{K_0}(\bar{c}(m-q_1u)\ov{q_1};{q_0})S(r\ov u,\pm n;\frac{rcq_0}{n_1})\mcW_{\pm}(\frac{m}{cq/X},\frac{n_1^2n}{c^3q^3_0r/X}),
\end{multline*}
which corresponds to the sum with $(q_0,n_1)=1$
and 
\begin{multline*}
       \mathrm{Err}_3=\frac{X}{Cq^{1/2}}\sum_{\pm}\sum_\stacksum{c\leq  2C}{(c,q)=1}\sum_{n_1|rc}\sum_{n\geq 1}\frac{\lambda(n,n_1{q_0})}{nn_1{q_0}} \sum_{m\in \mathbb{Z}}\widehat{K_1}(\bar{c}m\ov{q_0};q_1)S(rq_1\ov{m},\pm n;\frac{rc}{n_1})\\
       \times \sumstar_{u(q_0)}\widehat{K_0}(\bar{c}(m-q_1u)\ov{q_1};q_0)\mcW_{\pm}(\frac{m}{cq/X},\frac{n_1^2n}{c^3q_0r/X}),
\end{multline*}
that  corresponds to the complementary sum with $q_0|n_1$ in \eqref{eqaftervoronoi}.
\subsection{The case of $\mathrm{Err}_3$}
Notice that we have
$$\sumstar_{u(q_0)}\widehat{K_0}(\bar{c}(m-q_1u)\ov{q_1};q_0)=-\widehat{K_0}(\bar{c}m\ov{q_1};q_0)\ll \|\widehat{K_0}\|_{\infty}.$$
We can bound $\mathrm{Err}_3$ as follows
\begin{multline}\label{bound-err3}
       \mathrm{Err}_3\ll X^{o(1)}\frac{X\|\widehat{K_0}\|_{\infty}\|\widehat{K_1}\|_{\infty}}{Cq^{1/2}}\sum_{c\leq  2C}\sum_{n_1|rc}\sum_\stacksum{n\geq 1}{n^2_1n\leq C^3{q_0}r/X}\frac{|\lambda(n,n_1{q_0})|}{nn_1{q_0}} \sum_{m\leq ZCq/X}\bigl(rq_1\ov{m}, n;\frac{rc}{n_1}\bigr)^{1/2}\bigl(\frac{rc}{n_1}\bigr)^{1/2}\\
       \ll X^{o(1)}\frac{X\|\widehat{K_0}\|_{\infty}\|\widehat{K_1}\|_{\infty}{q_0}^{\varpi_3}}{Cq^{1/2}{q_0}}C\frac{ZCq}{X}\bigl(rC\bigr)^{1/2}\ll X^{o(1)}\|\widehat{K_0}\|_{\infty}\|\widehat{K_1}\|_{\infty}Zr^{1/2}X^{3/4}{q_0}^{\varpi_3-7/4}q^{1/2}
\end{multline}
upon plugging in $C=(X/q_0)^{1/2}$. Here $\varpi_3=5/14$ is the Kim--Sarnak bound \cite{KimSar}.

\subsection{The case of $\mathrm{Main}_{0}$}
For $\mathrm{Main}_{0}$, since $n_1|rc$, we have 
$$S(r\ov u,\pm n;\frac{rcq_0}{n_1})=S(\ov{c}n_1\ov u,\pm \ov{rc}n_1n;{q_0})
S(\ov{q_0}r\ov u,\pm \ov{q_0}n;{rc}/n_1)$$ so that we can write the sum over $u\bmod c{q_0}$ in $\mathrm{Main}_{0}$ as 
$$
       S(\ov{q_0}{q_1}r\ov{m},\pm \ov{q_0}n;{rc}/n_1)\widehat{K_1}(\bar{c}m\ov{q_0};{q_1})\sumstar_{u(q_0)}\widehat{K_0}(\bar{c}(m-q_1u)\ov{q_1};{q_0})S(\ov{c}n_1\ov u,\pm \ov{rc}n_1n;{q_0}).
$$
Hence we have
\begin{multline}
\mathrm{Main}_{0}=\frac{X{q_0}}{Cq^{1/2}}\sum_{\pm}\sum_\stacksum{c\leq  2C}{(c,q)=1}\sum_\stacksum{n_1|rc}{(n_1,q_0)=1}\sum_{n\geq 1}\frac{\lambda(n,n_1)}{nn_1}\\
\times\sum_{m\in\mathbb{Z}}\widehat{K_1}(\bar{c}m\ov{q_0};{q_1})N_{c,r}(m,n;{q_0})S(\ov{q_0}{q_1}r\ov{m},\pm \ov{q_0}n;{rc}/n_1)\mcW_{\pm}(\frac{m}{cq/X},\frac{n_1^2n}{c^3q^3_0r/X}),\label{n1qcoprime}
\end{multline}
where
\begin{multline}\label{defn-N}
N_{c,r}(m,n;{q_0})=\frac{1}{{q_0}}\sumstar_{u(q_0)} \widehat{K_0}(\bar{c}(m-q_1u)\ov{q_1};{q_0})S(\ov{c}n_1\ov u,\pm \ov{rc}n_1n;{q_0})\\
=\frac{1}{{q_0^{1/2}}}\sumstar_{u(q_0)} \widehat{K_0}(\bar{c}(m-q_1u)\ov{q_1};{q_0}) \Kl_2(\pm \ov{c}^2\ov{r}n^2_1n\ov u;{q_0}).
\end{multline}

We break the $c$-sum in \eqref{n1qcoprime} into $O(\log X)$ many dyadic intervals with $c\sim C'$, where $C'$ satisfies
$$ C'\leq 2C=2(X/{q_0})^{1/2}.$$ 
\begin{notation} To lighten the expressions to come we write 
$$A\lesssim B\hbox{ in place of }A\leq X^{o(1)}B.$$
\end{notation}

By \eqref{truncate-m} and  \eqref{decay-of-W},  we know that the $m$ and $(n_1,n)$ sums in \eqref{n1qcoprime} can be truncated at
\begin{equation}\label{truncation-of-n}
m\lesssim M=Z\frac{{C'}q}{X}, nn^2_1\lesssim \frac{{C'}^3q^{3}_0r}{X}.
\end{equation}

For each fixed $n_1$ we break the $n$-sum in \eqref{n1qcoprime} into $O(\log q)$ dyadic intervals $n\sim N/n^2_1$ with $N$ satisfying $$N\lesssim \frac{{C'}^3q^{3}_0r}{X}.$$

Now for each $c\sim C'$ and $nn^2_1\sim N$ we will evaluate the truncated version of $\mathrm{Main}_{0}$.

\subsubsection{Cauchy--Schwarz}\label{CSsec}
We now factorize $c=c_1c_2$ with $$c_1\leq C',\ {n_1|rc_1},\ c_1|(n_1r)^\infty\hbox{ and }(c_2,n_1r)=1.$$
Then we apply Cauchy--Schwarz inequality and the Rankin--Selberg estimate to bound  the sum $\mathrm{Main}_{0}$ as follows (for the various choices of $\pm$)
\begin{equation}\label{main-after-cauchy}
\mathrm{Main}_{0}\ll  X^{o(1)}\frac{Xq_0}{Cq^{1/2}}\sup_{N\lesssim \frac{{C'}^3q^{3}_0r}{X}}\frac{1}{N^{1/2}}B(N)^{1/2}
\end{equation}
with
\begin{multline*}
B(N):=\sumsum_\stacksum{c_1,nn_1^2\approx N}{(n_1,q_0)=1}n_1\biggl| \sum_{m\leq M}\sum_\stacksum{c_2\sim C'/c_1}{(c_2,q)=1}\widehat{K_1}(\bar{c}m\ov{q_0};{q_1})\\
\times
S(\ov{q_0}{q_1}r\ov{m},\pm \ov{q_0}n;{rc_1c_2}/n_1)N_{c_1c_2,r}(m,n;{q_0})\mcW_{\pm}(\frac{m}{c_1c_2q/X},\frac{n_1^2n}{c_1^3c_2^3q^3_0r/X}) \biggr|^2 U\left(\frac{n}{N/n^2_1}\right).
\end{multline*}
Here $U\in \mcC^\infty_c(\Rr_{>0})$ is a smooth function satisfying $U^{(j)}(x)\ll_j 1$ for $j\geq 0$.

After opening the square, the factor $B(N)$ equals
\begin{gather}\label{eqprepoisson}
B(N)=\sumsum_{c_1,n_1}n_1\sumsum_{m,m'}\sumsum_{c_2,c_2'}\widehat{K_1}(\bar{c}m\ov{q_0};{q_1})\ov{\widehat{K_1}(\ov{c'}m'\ov{q_0};{q_1})}\\
\times\sum_{n\geq 1} S(\ov{q_0}{q_1}r\ov{m},\pm \ov{q_0}n;{rc}/n_1)\ov{S(\ov{q_0}{q_1}r\ov{m'},\pm \ov{q_0}n;{rc'}/n_1)}\nonumber\\
\times N_{c,r}(m,n;{q_0})\ov{N_{c',r}(m',n;{q_0})}\mcW\left(\frac{n}{N/n^2_1}\right),	\nonumber
\end{gather}
 where 
 \begin{equation}
     \label{cc'def}
     c=c_1c_2, c'=c_1c'_2
 \end{equation}
 and
 \begin{equation*}\label{weight-before-poisson}
 \mcW\left(\frac{n}{N/n^2_1}\right)=U\left(\frac{n}{N/n^2_1}\right)\mcW_{\pm}(\frac{m}{cq/X},\frac{n_1^2n}{c^3q^3_0r/X})\ov{\mcW_{\pm}}(\frac{m'}{{c'}q/X},\frac{n_1^2n}{{c'}^3q^3_0r/X}).
 \end{equation*}

We apply Poisson formula to the $n$-variable keeping in mind that 
\begin{equation*}\label{eq-MMNN} 
n\mapsto S(\ov{q_0}{q_1}r\ov{m},\pm \ov{q_0}n;{rc}/n_1)\ov{S(\ov{q_0}{q_1}r\ov{m'},\pm \ov{q_0}n;{rc'}/n_1)}N_{c,r}(m,n;{q_0})\ov{N_{c',r}(m',n;{q_0})}
\end{equation*}
 is periodic of period $q_0k$ with $k=rc_1c_2c_2'/n_1$, and we see that \eqref{eqprepoisson} becomes
\begin{multline}\label{eqpostpoisson}
B(N)=\sumsum_{c_1,n_1}n_1\sumsum_{m,m'}\sumsum_{c_2,c_2'}\widehat{K_1}(\bar{c}m\ov{q_0};{q_1})\ov{\widehat{K_1}(\ov{c'}m'\ov{q_0};{q_1})}\\
\times\frac{N}{n^2_1 {q_0}k}\sum_{n\in\Zz}\mathrm{FT}_1(n;q_0)\mathrm{FT}_2(n;k)\what \mcW(n/N^*),
\end{multline}
 where (after inserting \eqref{defn-N})
\begin{multline}\label{eq-qsum2}
\mathrm{FT}_1(n;q_0)=\frac{1}{{q_0}}\sumsum_{u,u'\bmod  {q_0}}
\widehat{K_0}(\bar{c}(m-q_1u)\ov{q_1};{q_0})\ov{\widehat{K_0}(\ov{c'}(m'-q_1u')\ov{q_1};{q_0})} \\
\times\sum_{v\bmod  {q_0}}\Kl_2(\pm \ov{c}^2\ov{r}n^2_1v\ov u;{q_0})
\ov{\Kl_2(\pm \ov{c'}^2\ov{r}n^2_1v\ov{u'};{q_0})} \, e\left(\frac{nv\bar{k}}{q_0}\right)
\end{multline}
and with 
\begin{equation}
    \label{kdef}k=rc_1c_2c'_2/n_1
\end{equation}and
$$\FT_2(n;k)=\sum_{v(k)}S(\ov{q_0}{q_1}r\ov{m},\pm \ov{q_0}v;{rc_1c_2}/n_1)\ov{S(\ov{q_0}{q_1}r\ov{m'},\pm \ov{q_0}v;{rc_1c'_2}/n_1)}\, e\left(\frac{ nv \ov{q_0}}k\right)$$
and
     \begin{equation}\label{defN*}
       N^*:=q_0kn_1^2/ N.
       \end{equation}
    The Fourier transform $\what \mcW(y)$ of $\mcW$ is of rapid decay when $|y|\gg q^{\varepsilon}$. Hence we can truncate the dual $n$-sum in \eqref{eqpostpoisson} at $|n|\ll q^{\varepsilon} N^*$. For $|n|\ll q^{\varepsilon} N^*$, from the estimate \eqref{Jacquet-Shalika}, we readily have the bound 
    \begin{equation}\label{scalar-of-W}
    \what \mcW(n/N^*)\ll  \frac{N}{{C'}^3q^{3}_0r/X}.
    \end{equation}

\subsubsection{Computation of $\FT_2(n;k)$}\label{computation-fourier-transform}
The following bounds can be proved.
\begin{lemma}\label{bound-k-part}
 We have the following estimates 
\begin{equation*}
\FT_2(0;k)\ll \delta_{c_2=c'_2}\frac{r{c_1{c_2}{c'_2}}}{n_1}\sum_{d|(rc_1c_2/n_1,r(m-m'))}d
\end{equation*}
and for $n\not\equiv 0\mods{k}$
\begin{equation*}
\begin{split}
\FT_2(n;k)\ll (\frac{r{c_1}}{n_1})^2 c_2c'_2.
\end{split}\end{equation*}
\end{lemma}
\proof Inserting the definition of the Kloosterman sum, we can write
\begin{equation*}
\FT_2(n;k)=\frac{r{c_1{c_2}{c'_2}}}{n_1}\sumstar_{y_1(rc_1c_2/{n_1})}e\left(\frac{\ov{y_1}\ov{q_0}{q_1}r\ov{m}}{rc_1c_2/{n_1}}\right)\sumstar_\stacksum{y_2(rc_1c'_2/{n_1})}{\pm(y_1c'_2-y_2c_2)+n\equiv 0(rc_1c_2c'_2/n_1)}e\left(\frac{-\ov{y_2}\ov{q_0}{q_1}r\ov{m'}}{rc_1c'_2/{n_1}}\right).
\end{equation*}

We consider two cases.
\par Case 1. If $n\equiv 0 \mods{k}$, then the inner sum vanishes unless $c_2=c'_2$ and $y_1\equiv y_2 \bmod rc_1c_2/n_1$ in which case we obtain that
\begin{equation*}
\begin{split}
\FT_2(0;k)=&\delta_{c_2=c'_2}\frac{r{c_1{c_2}{c'_2}}}{n_1}\sumstar_{y_1(rc_1c_2/{n_1})}e\left(\frac{\ov{y_1}(\ov{m}-\ov{m'})\ov{q_0}{q_1}r}{rc_1c_2/{n_1}}\right)\\
=&\delta_{c_2=c'_2}\frac{r{c_1{c_2}{c'_2}}}{n_1}\sum_{d|(rc_1c_2/n_1,r(m-m'))}d\mu\left(\frac{rc_1c_2/n_1}{d}\right).
\end{split}\end{equation*}

\par Case 2. If $n\neq 0 \mods{k}$, then we can further write
\begin{multline*}
\FT_2(n;k)=\frac{r{c_1}{c_2}{c'_2}}{n_1}\sumstar_{y_1(rc_1/{n_1})}e\left(\frac{\ov{y_1}\ov{q_0}{q_1}r\ov{m}\ov{c_2}}{rc_1/{n_1}}\right)\sumstar_\stacksum{y_2(rc_1/{n_1})}{\pm(y_1c'_2-y_2c_2)+n\equiv 0(rc_1/n_1)}e\left(\frac{-\ov{y_2}\ov{q_0}{q_1}r\ov{m'}\ov{c'_2}}{rc_1/{n_1}}\right)\\
\times \sumstar_{x_1(c_2)}e\left(\frac{\ov{x_1}\ov{q_0}{q_1}\ov{m}\ov{c_1}n_1}{c_2}\right)\sumstar_\stacksum{x_2(c'_2)}{\pm(x_1c'_2-x_2c_2)+n\equiv 0(c_2c'_2)}e\left(\frac{-\ov{x_2}\ov{q_0}{q_1}\ov{m'}\ov{c_1}n_1}{c'_2}\right)\\
\ll \frac{r{c_1}{c_2}{c'_2}}{n_1}\varphi(\frac{r{c_1}}{n_1})
\end{multline*}
with $\varphi$ being Euler's totient function.
\qed

\subsubsection{Computation of $\FT_1(n;{q_0})$}    \label{secq-sum}

Using  calculations very similar to \cite{LMS}*{\S 4.2.2} (in the notation of \cite{LMS}*{(4.21)-(4.22)}, $\widehat{L}(x)=K_0(xc;q_0)e(m\ov{q_1}x/q_0)$), we can express $\mathrm{FT}_1(n;q_0)$ which is given
in \eqref{eq-qsum2} as $$
\mathrm{FT}_1(n;q_0)=
\sum_{v\bmod  {q_0}}Z(v)\ov{Z'(v-\delta)},$$
where
\begin{equation*}
Z(v)=Z_{\alpha,\beta}(v):=\frac{1}{\sqrt{q_0}}\sum_{x\in\Fqqt}K_0(xv;{q_0})e\left(\frac{\alpha xv}{{q_0}}\right)\Kl_2(\beta x;{q_0}),	
\end{equation*}
and $Z'(v)$ is defined likewise with the parameter $(\alpha,\beta)$ being replaced by $(\alpha',\beta')$ and $\alpha,\alpha',\beta,\beta'\in\Ff_{q_0}$ are given by
\begin{gather}\alpha= \ov{c} m\ov{q_1},\  \alpha'=\ov{c'} m'\ov{q_1},\ \beta=\eps \ov{c}^3\ov{r} n_1^2,\, \beta'=\eps\ov{c'}^3\ov{r} n_1^2,\ \delta=\ov k n,\label{eq-actual}
\end{gather}
with $\eps=\pm 1$.

The following result is analogous to \cite{LMS}*{Prop 4.5}.

Write $q:=q_0$; let $T_\mcF(\Fq)$ be the subgroup of $\Fqt$ defined by
$$ T_\mcF(\Fq)=\{\lambda\in\Fqt,\ [\times\lambda]^*\mcF\hbox{ is geometrically isomorphic to }\mcF\}
$$
and let $\mathrm{Aff}_\mcF(\Fq)$ be the subgroup of affine linear transformations of $\Fq$ defined by
$$ \mathrm{Aff}_\mcF(\Fq)=\{\gamma:y\mapsto ay+b,\ [\gamma]^*\mcF\hbox{ is geometrically isomorphic to }\mcF\}.
$$

\begin{proposition} \label{sqrootcancel1}
 Assuming that the sheaf $\mcF$ is good in the sense of Definition \ref{gooddef} and that $q$ is sufficiently large depending on the complexity $C(\mcF)$. For any $\alpha,\beta,\alpha',\beta',\delta\in\Fqt$, we have
\begin{equation}
	\label{ZZmomentdelta1}
	\sum_{v}Z(v)\ov{Z'(v-\delta)}=O(q^{1/2}).
\end{equation}
If $\delta=0$, \eqref{ZZmomentdelta1} holds unless one of the following holds
\begin{itemize}
	\item $\alpha/\alpha'=\beta/\beta'$ and
	$$\gamma:=\alpha/\alpha'=\beta/\beta'\in T_\mcF(\Fq),$$
	\item $\alpha/\alpha'\not=\beta/\beta'$, $\beta/\beta'\not=1$ and
	$$\gamma: y\mapsto \frac{\beta}{\beta'}y+\alpha'(\frac{\alpha}{\alpha'}-\frac{\beta}{\beta'})\in \mathrm{Aff}_\mcF(\Fq).$$
\end{itemize}
In these two cases, we have
\begin{equation}
	\label{ZZmoment1}
	\sum_{v}Z(v)\ov{Z'(v)}=c_\mcF(\gamma)q+O(q^{1/2})
\end{equation}
for $c_\mcF(\gamma)$ some complex number of modulus $1$. In these estimates, the implicit constants depend only on $C(\mcF)$.
\end{proposition}

%\begin{remark}
%    It follows from the first %condition of Definition \ref{gooddef} %that $|\mathrm{Aff}_\mcF(\Fq)|$ is %bounded only in terms of the %complexity $C(\mcF)$.
%\end{remark}

Returning to our original sum, applying Proposition \ref{sqrootcancel1} (with $q=q_0$) and using that \eqref{eq-actual}, we see that the $\FT_1(n;{q_0})$ in \eqref{eq-qsum2} is  $O\big({q_0^{1/2}}\big)$ unless $\delta=0$ and either
\begin{equation}\label{delta=0}\gamma:={c'}m/cm'={c'}^3/c^3\in T_\mcF({\mathbf{F}_{q_0}})
\end{equation}
or 
$$\gamma: y\mapsto \frac{{c'}^3}{c^3}y+\ov{c'}m'\ov{q_1}(\frac{c'm}{cm'}-\frac{{c'}^3}{c^3})\in \mathrm{Aff}_\mcF({\mathbf{F}_{q_0}})$$
in which case 
$\FT_1(n;{q_0})$  equals $c_\mcF(\gamma){q_0}+O\big({q_0^{1/2}}\big)$ with $|c_\mcF(\gamma)|=1$.

\subsection{Contribution of the $n=0$ frequency}\label{contribution-of-zero}

In this section we bound the contribution to $B(N)$ in \eqref{eqpostpoisson}  from the frequency $n=0$ , given by
\begin{multline*}
B_{n=0}(N)=\sumsum_{c_1,n_1}\sumsum_{m,m'}\sumsum_{c_2,c_2'}\widehat{K_1}(\bar{c}m\ov{q_0};{q_1})\ov{\widehat{K_1}(\ov{c'}m'\ov{q_0};{q_1})}\\
\times\frac{N}{ {q_0}rc_1c_2c_2'}\mathrm{FT}_1(0;q_0)\mathrm{FT}_2(0;rc_1c_2c_2'/n_1)\what \mcW(0).
\end{multline*}

 By Lemma \ref{bound-k-part}, when $n=0$ we know that
\begin{equation}\label{cequal}
c_2=c_2',\ c=c',\ k=rc_1c_2^2/n_1.	
\end{equation}

We use the case $\delta=0$ of Proposition \ref{sqrootcancel1}: by \eqref{cequal} and \eqref{delta=0} we have that $\mathrm{FT}_1(0;q_0)$ is  $O({q_0^{1/2}})$ unless we have the congruence modulo $q_0$:
$$m\equiv m'\mods{q_0}$$ in which case 
$\mathrm{FT}_1(0;q_0)$ equals $c{q_0}+O({q_0^{1/2}})$ for some complex number $c$ of modulus $1$.

Similar to the calculations presented in \cite{LMS}*{\S 5}, upon plugging in Proposition \ref{sqrootcancel1} and Lemma \ref{bound-k-part}, we see that $B_{n=0}(N)$ is bounded by 
\begin{multline*}
\|\widehat{K_1}\|^2_{\infty}\sumsum_{c_1,n_1}\sumsum_{m,m'}\sum_{c_2}\frac{N}{ {q_0}rc_1c^2_2}\bigl(q_0\delta_{m\equiv m'\bmod {q_0}}+{q_0^{1/2}}\bigr) \frac{r{c_1{c^2_2}}}{n_1}\sum_{d|(rc_1c_2/n_1,r(m-m'))}d\\
\ll X^{o(1)}\|\widehat{K_1}\|^2_{\infty}
rN\bigl({C'}^2M+\frac{1}{{q_0^{1/2}}}{C'}M^2\bigr).
\end{multline*}
Taking the square root of this term and multiplying it by $\frac{Xq_0}{Cq^{1/2}}\frac{1}{N^{1/2}}$, where $N$ runs over all $$N\ll \frac{{C'}^3q^{3}_0r}{X},$$ we see that the contribution of these terms to \eqref{main-after-cauchy} and therefore to \eqref{n1qcoprime} is bounded by
\begin{multline}\label{n=0bound}
X^{o(1)}\|\widehat{K_1}\|_{\infty}\frac{X{q_0}}{Cq^{1/2}}r^{1/2}\bigl({C'}M^{1/2}+{C'}^{1/2}M/{q_0}^{1/4}\bigr)\\
\ll X^{o(1)}\|\widehat{K_1}\|_{\infty} r^{1/2}\bigl(Z^{1/2}C^{1/2}X^{1/2}{q_0}+ZC^{1/2}q^{1/2}{q_0}^{3/4}\bigr)\\
\ll X^{o(1)}\|\widehat{K_1}\|_{\infty}r^{1/2}\bigl(Z^{1/2} X^{3/4}{q_0}^{3/4}+ZX^{1/4}q^{1/2}{q_0^{1/2}}\bigr).
\end{multline}

\subsection{Contribution from the $n\not=0$ frequencies}  \label{secn=0}

Recall from \eqref{eqpostpoisson} the contribution from the non-zero frequencies $n\neq 0$ to $B(N)$ is given by
\begin{multline*}
B_{n\neq 0}(N)=\frac{N}{{q_0}r}\sumsum_{c_1,n_1}\frac1{c_1}\sum_{m\leq M}\sum_{m'\leq M}\sum_{c_2\sim C'/c_1}\sum_{c_2'\sim C'/c_1}\frac1{c_2c'_2}\widehat{K_1}(\ov{c_1c_2}m\ov{q_0};{q_1})\ov{\widehat{K_1}(\ov{c_1c'_2}m'\ov{q_0};{q_1})}\\
\times\sum_{n\not=0}\mathrm{FT}_1(n;{q_0})\mathrm{FT}_2(n;k)\what \mcW(n/N^*),
\end{multline*}
where $k=rc_1c_2c'_2/n_1$ and $N^*=q_0kn_1^2/ N$.

We distinguish two cases: 
\vskip 0.5cm

-- If $n\equiv0\bmod  {q_0}$ we use again Proposition \ref{sqrootcancel1}  for $\delta=0$ and using  Lemma \ref{bound-k-part} and \eqref{scalar-of-W}, we obtain the following bounds 
$$
B_{{q_0}|n,n\not=0}(N)\ll X^{o(1)}\|\widehat{K_1}\|^2_{\infty}r^{2}{C'}^4{q_0^{1/2}}M^2\frac1{q_0^{1/2}}\cdot  \frac{N}{{C'}^3q^{3}_0r/X}.
$$

-- If $n\not\equiv 0\bmod  {q_0}$ we apply \eqref{ZZmomentdelta1} of Proposition \ref{sqrootcancel1}  for $\delta\not=0$ and obtain 
$$B_{n\not\equiv 0 \bmod  {q_0}}(N)\ll X^{o(1)}\|\widehat{K_1}\|^2_{\infty}r^{2}{C'}^4{q_0^{1/2}}M^2\cdot    \frac{N}{{C'}^3q^{3}_0r/X}.$$

 The reader is referred to \cite{LMS}*{\S 6.1--\S 6.2} for very similar calculations.

The non-zero frequencies contribution to \eqref{main-after-cauchy} and hence to \eqref{n1qcoprime} is bounded by 
\begin{gather}\nonumber
X^{o(1)} \frac{Xq_0}{Cq^{1/2}} \sup_{N\lesssim \frac{{C'}^3q^{3}_0r}{X}}\frac{1}{N^{1/2}} \left(B_{n\not\equiv 0 \bmod  {q_0}}(N)+B_{{q_0}|n,n\not=0}(N)\right)^{1/2}\\
\ll X^{o(1)} \frac{Xq_0}{Cq^{1/2}} \sup_{N\lesssim \frac{{C'}^3q^{3}_0r}{X}}\frac{1}{N^{1/2}}   \left(\|\widehat{K_1}\|^2_{\infty}r^{2}{C'}^4{q_0^{1/2}}M^2    \frac{N}{{C'}^3q^{3}_0r/X}\right)^{1/2}\nonumber\\
\ll X^{o(1)} Z\|\widehat{K_1}\|_{\infty}r^{1/2}\frac{X^{1/2}C^{1/2}q^{1/2}}{q_0^{1/4}}\ll X^{o(1)} Z\|\widehat{K_1}\|_{\infty}r^{1/2}\frac{X^{3/4}q^{1/2}}{q_0^{1/2}},\label{nonzerocomb}
\end{gather}
since by \eqref{truncation-of-n} $M=Z\frac{{C'}q}{X}$ and $C'\leq 2C=2(X/q_0)^{1/2}$.

\subsection{Bounding $\mcC_{V,r}(X;K)$: the final steps}\label{q-divide-n1}
Let us recall that the sum $\mathrm{Main}$ in \eqref{eqaftervoronoi} was split into two subsums $\mathrm{Main}_0$ and $\mathrm{Err}_3$ depending on whether $(n_1,q_0)=1$ or not. 

By \eqref{n=0bound} and \eqref{nonzerocomb} the sum $\mathrm{Main}_0$ in \eqref{n1qcoprime} is bounded by
\begin{equation}\label{n1qcoprimefinal}
\begin{split}
\ll& X^{o(1)} \|\widehat{K_1}\|_{\infty}r^{1/2}\Big(Z^{1/2}X^{3/4}{q_0}^{3/4}+ZX^{1/4}q^{1/2}{q_0^{1/2}}+Z\frac{X^{3/4}q^{1/2}}{q_0^{1/2}}\Big).	\end{split}
\end{equation}
According to \eqref{bound-err3},  the complement sum $\mathrm{Err}_3$ is bounded by \begin{equation}\label{qdivn1final}
\ll X^{o(1)}\|\widehat{K_0}\|_{\infty}\|\widehat{K_1}\|_{\infty}Zr^{1/2}X^{3/4}{q_0}^{\varpi_3-7/4}q^{1/2},
\end{equation}
where $\varpi_3=5/14$ is the Kim--Sarnak bound \cite{KimSar}. This bound is absorbed into the last factor in \eqref{n1qcoprimefinal} (by simply applying the trivial bound $\|\widehat{K_0}\|_{\infty}\ll {q_0^{1/2}}$).

Combining the bound \eqref{qdivn1final} with \eqref{n1qcoprimefinal} we see that the sum $\mathrm{Main}$ in \eqref{S'sum} and hence the sum $\mcC_{V,r}(X;K)$ in \eqref{Sdef} is bounded as follows
\begin{equation*}
\mcC_{V,r}(X;K)\ll X^{o(1)} Z\|\widehat{K_1}\|_{\infty}r^{1/2}\Big(X^{3/4}{q_0}^{3/4}+X^{1/4}q^{1/2}{q_0^{1/2}}+\frac{X^{3/4}q^{1/2}}{q_0^{1/2}}\Big). 	
\end{equation*}
This completes the proof of Theorem \ref{mainthm2}.

\section{Square-root cancellation for certain exponential sums}
\label{secsqroot}

In this section we establish Proposition \ref{sqrootcancel1} whose statement and notation we recall below:

Given $q$ a prime and $K(\bullet)$ the trace function of a geometrically irreducible middle-extension $\ell$-adic sheaf $\mcF$ on $\Aa^1_\Fq$ weight $\leq 0$ and complexity $C(\mcF)$; We will assume that the sheaf $\mcF$ is Fourier and {\em good} that is

-- The local monodromy of $\mcF$ at infinity has no indecomposable summand with slope equal to $1$.

Given $\alpha,\beta, \alpha',\beta'\in\Fqt$, we set
\begin{equation}\label{Zcompute}
Z(v):=Z_{\alpha,\beta}(v):=\frac{1}{\sqrt{q}}\sum_{x\in\Fqt}K(xv)e_q({\alpha xv})\Kl_2(\beta x;{q}),	
\end{equation}
and define 
$$Z'(v):=Z_{\alpha',\beta'}(v).$$

Let $T_\mcF(\Fq)$ be the subgroup of $\Fqt$ defined by
$$ T_\mcF(\Fq)=\{\lambda\in\Fqt,\ [\times\lambda]^*\mcF\hbox{ is geometrically isomorphic to }\mcF\}
$$
and let $\mathrm{Aff}_\mcF(\Fq)\supset T_\mcF(\Fq)$ be the subgroup of affine linear transformations of $\Fq$ defined by
$$ \mathrm{Aff}_\mcF(\Fq)=\{\gamma:y\mapsto ay+b,\ (a,b)\in\Fqt\times\Fq\ [\gamma]^*\mcF\hbox{ is geometrically isomorphic to }\mcF\}.
$$

\begin{proposition}\label{sqrootcancel2} 
 Assuming that the sheaf $\mcF$ is good and that $q$ is sufficiently large depending on $C(\mcF)$. For any $\alpha,\beta,\alpha',\beta',\delta\in\Fqt$, we have
\begin{equation}
	\label{ZZmomentdelta}
	\sum_{v}Z(v)\ov{Z'(v-\delta)}=O(q^{1/2}).
\end{equation}
If $\delta=0$ the above bound holds unless one of the following holds
\begin{itemize}
	\item $\alpha/\alpha'=\beta/\beta'$ and
	$$\gamma:=\alpha/\alpha'=\beta/\beta'\in T_\mcF(\Fq),$$
	\item $\alpha/\alpha'\not=\beta/\beta'$, $\beta/\beta'\not=1$ and
	$$\gamma: y\mapsto \frac{\beta}{\beta'}y+\alpha'(\frac{\alpha}{\alpha'}-\frac{\beta}{\beta'})\in \mathrm{Aff}_\mcF(\Fq).$$
\end{itemize}
In these two cases we have
\begin{equation}
	\label{ZZmoment}
	\sum_{v}Z(v)\ov{Z'(v)}=c_\mcF(\gamma)q+O(q^{1/2})
\end{equation}
for $c_\mcF(\gamma)$ some complex number of modulus $1$. In these estimates, the implicit constants depend only on $C(\mcF)$.
\end{proposition}

\begin{remark}This proposition is analogous to  \cite{LMS}*{Prop 4.5}
    where we studied the correlations of  sums similar to $Z(v),Z'(v)$ but with $e_q({\alpha xv})$ replaced by the Kloosterman sum $\Kl_2(\alpha xv;q)$.
\end{remark}

\subsection{The case $\delta=0$}
We observe that for $v\in \Fqt$ one has
$$Z(v)=\frac{1}{\sqrt{q}}\sum_{x\in\Fqt}K(x)e_q({\alpha x})\Kl_2(\beta ux;{q})=\frac{1}{\sqrt{q}}\sum_{y\in\Fq}\what K(y+\alpha)e_q(\beta /y v u)^{-1}$$
where $u=-v^{-1}$.
In particular
$$\sum_{v\in\Fqt}Z(v)\ov{Z'(v)}=\sumsum_{x,x'\in \Fqt}K(x)\ov{K(x')}e_q(\alpha x-\alpha' x')\frac{1}q\sum_{u\in\Fqt}\Kl_2(\beta xu;{q})\ov{\Kl_2(\beta' x'u;{q})}$$
The inner sum equals (since $\Kl_2(0;q)=0$)
\begin{eqnarray*}
    \frac{1}q\sum_{u\in\Fqt}\Kl_2(\beta xu;{q})\ov{\Kl_2(\beta' x'u;{q})}&=&\frac{1}q\sum_{u\in\Fq}\Kl_2(\beta ux;{q})\ov{\Kl_2(\beta' x'u;{q})}
    \\&=&\frac{1}q\sum_{z,z'\in\Fqt}e_q(\beta x/z-\beta' x'/z')\frac{1}{q}\sum_{u\in\Fq}e_q((z-z')u)\\
    &=&\frac{1}q\sum_{z\in\Fqt}e_q((\beta x-\beta' x')/z)\\
    &=&\delta_{\beta x=\beta'x'}-\frac{1}q
\end{eqnarray*}
Hence
\begin{align*}
	\sum_{v\in\Fqt}Z(v)\ov{Z'(v)}&=\sum_{x\in \Fqt}K(x)\ov{K(\frac{\beta}{\beta'}x)}e_q({(\alpha -\frac{\alpha'\beta}{\beta'})x })+(\what K(\alpha)-K(0))\ov( \what K(\alpha')-K(0)) \\
	&=\sum_{x\in \Fqt}K(x)\ov{K(\frac{\beta}{\beta'}x)}e_q({(\alpha -\frac{\alpha'\beta}{\beta'})x })+O(1)
\end{align*}
since $\mcF$ is Fourier and therefore $\|\what K\|_\infty\ll_{C(\mcF)}1$.

The $x$-sum in the last expression above is $O(q^{1/2})$ unless one has a geometric isomorphism 
$$\mcF\otimes\mcL_{e_q(\alpha'\eta x)}\simeq [\times \beta/\beta']\mcF$$
for $\eta=\alpha/\alpha' -\beta/\beta'$.

Suppose $\eta=0$ then, since $\mcF$ is geometrically irreducible, we then must have $\alpha/\alpha'=\beta/\beta'\in T_\mcL(\Fq)$ and \eqref{ZZmoment} holds.

If $\eta\not=0$, taking the Fourier transforms on both sides we have
$$[+\alpha'\eta]\what\mcF\simeq [\times \beta'/\beta]\what\mcF\Longleftrightarrow  [\gamma]^* \mcF\simeq \mcF$$
where $\gamma$ is the affine transformation 
$$\gamma: y\mapsto \frac{\beta}{\beta'}y+\alpha'(\frac{\alpha}{\alpha'}-\frac{\beta}{\beta'}).$$
Notice that if $\beta=\beta'$, then $\gamma$ would then be a non-trivial translation which would imply (for $C(\mcF)\ll_q 1$) that $\what\mcF$ is an Artin--Schreier sheaf which is excluded.

\subsection{The case $\delta\not=0$}

Let $\mcL_\alpha$ be the Artin--Schreier sheaf associated with the additive character $x\mapsto e_q(\alpha x)$, then
$$K_\alpha: x\mapsto K(x)e_q(\alpha x)$$ is the trace function of the twist $\mcF_\alpha:=\mcF\otimes\mcL_\alpha.$

Next we observe that $Z(v)$ is the multiplicative convolution of $x\mapsto K_\alpha(x)$ with $x\mapsto \Kl_2(\beta x^{-1})$ so that
$Z(v)$ is the trace function of the middle  convolution sheaf
\begin{equation}\label{mcZdef}
\mcZ=\mcZ_{\alpha,\beta}:=\mcF_\alpha  \star [y\ra\beta/y]^*\HYPK_2=\mcF_\alpha\star\mcK	
\end{equation}
say. Likewise $Z'$ is the trace function of
$$
\mcZ'=\mcZ_{\alpha',\beta'}=\mcF_{\alpha'} \star [y\ra\beta'/y]^*\HYPK_2=\mcF_{\alpha'}\star\mcK'.
$$

Notice that  $\mcF_{\alpha}$ and $\mcF_{\alpha'}$ are geometrically irreducible and non trivial (being Artin--Schreier twists of the geometrically irreducible sheaf $\mcF$, which is Fourier). 

This implies that $\mcZ$ and $\mcZ'$ are also geometrically irreducible. This is because $\mcZ$ and $\mcZ'$ are obtained from $\mcF_{\alpha}$ and $\mcF_{\alpha'}$ by two applications of a geometric Fourier transform which preserves geometric irreducibility. 

It follows that \eqref{ZZmomentdelta} holds unless we have the geometric isomorphism
$$\mcZ'\simeq [+\delta]^*\mcZ.$$
To show that this does not occurs, we follow the same path as in \cite{LMS}*{\S 8} excepted that the present argument is much simpler. 

\begin{lemma}\label{he-lisse} 
If $\mcF$ has no slope at $\infty$ equal to $1$, then $\mcZ$ is lisse away from $\{0, \infty\}$. 
\end{lemma}

\begin{proof} The middle convolution $\mcF_\alpha \star \mcK $ is equal to the compactly supported convolution $\mcF_\alpha \star_{!} \mcK$ up to a lisse sheaf, so it suffices to prove this for the compactly supported convolution.

We can rewrite the  compactly supported convolution   as
$$\pi_{!}(\mcF_\alpha\otimes [\times\beta x/v]^*\KL_{2})$$
where
$$\pi:\map{\Gm\times\Gm}{\Gm}{(x,v)}{v}$$

To study where $\mcZ$ is lisse, we apply Deligne's semicontinuity theorem and examine the variation with $v$ of the Euler characteristic of the sheaf (in the $x$-variable) $\mcF_\alpha \otimes [x\ra\beta x/v]^*\HYPK_2$. Since $\HYPK_2$ is wildly ramified only at $\infty$ (with a single slope $1/2$) it suffice to look at the variation with $v$ of the Swan conductor
\begin{equation}
    \label{swanv}
    v\mapsto \swan_\infty(\mcF_\alpha\otimes [x\mapsto \beta x/v]\KL_2).
\end{equation}

We recall that for $\alpha\not=0$, $\mcL_\alpha$ has a (unique) slope at $\infty$, equal to $1$ and for  for any $v$ the sheaf $[x\mapsto \beta x/v]^*\KL_2$ has a single slope at $\infty$,  equal to  $1/2$, with multiplicity $2$.

Let $V$ be an indecomposable summand of the local monodromy representation of $\mcF$ at $\infty$ and $s\not=1$ be its slope; the slope of $(V  \otimes \mcL_{\alpha})$ equals $\max(s,1)\geq 1$; in particular it does not depend on $\alpha$.

It follows that the sheaf in the $x$-variable, $(V  \otimes \mcL_{\alpha}) \otimes [x\ra\beta x/v]^*\HYPK_2$ has for unique slope $\max(1/2,\max(s,1))=\max(s,1)\geq 1$ for every $v$ , so the contribution of this representation to the Swan conductor \eqref{swanv} is constant with $v$.

By Deligne's semicontinuity theorem $\mcF_\alpha\otimes\mcK$ is lisse on $\Gm$. \end{proof} 

\begin{lemma}\label{he-nontrivial}
If $\mcF_\infty$ has no irreducible summand with slope equal to $1$, then $\mcZ=\mcF_\alpha \star \mcK$ has a nontrivial singularity at zero. 
\end{lemma} 

\begin{proof} Following the same arguments as in the beginning of the proof \cite{LMS}*{Lemma 8.7} (the fact that the Euler characteristic of $\chi$ is non-positive) with the sheaf  $\mcF_\alpha=\mcF\otimes\mcL_\alpha$ here, replacing the sheaf denoted $\mcK=\mcF\otimes [\times\alpha]^*\KL_2$ and noting that the sheaf $\mcK=[x\mapsto \beta x^{-1}]^*\KL_2$ was denoted $\mcL$ in \cite{LMS}, one shows that at least of the following situations happens:

\begin{enumerate}

\item $\mcF$ is singular at some point on $\Gm$.

\item The local monodromy representation of $\mcF$ at zero is not unipotent.

\item The local monodromy representation of $\mcF$ at zero has a Jordan block of size $\geq 4$.

\item\label{fourthcase} The local monodromy representation of $\mcF$ at $\infty$ has a summand with slope $\geq 1/3$.  \end{enumerate}

Since $\mcF_\infty$ has no summand with slope $1$, the following holds for $\mcF_\alpha=\mcF \otimes \mcL_{\alpha}$. 

\begin{enumerate}

\item $\mcF_\alpha$ is singular at some point on $\Gm$.

\item The local monodromy representation of $\mcF_\alpha$ at zero is not unipotent.

\item The local monodromy representation of $\mcF_\alpha$ at zero has a Jordan block of size $\geq 4$.

\item The local monodromy representation of $\mcF_\alpha$ at $\infty$ has all its slopes $\geq 1$.
\end{enumerate}

The first three situations are the same as in \cite{LMS}*{Lemma 8.7} while the fourth is stronger in the present paper; hence the arguments of \cite{LMS}*{Lemma 8.7} apply and we obtain that $\mcF_\alpha \star \mcK $ has non-trivial local monodromy at zero. 
\end{proof}
We can now conclude the proof of \eqref{ZZmomentdelta}:
we have to show that $\mcZ'$ is not geometrically isomorphic to $[+\delta]\mcZ$ for any $\delta\not=0$ but if it were the case $\mcZ'$ would have a non-trivial singularity at $+\delta$ (inherited from the non-trivial singularity of $\mcZ$ at $0$ by Lemma \ref{he-nontrivial}) and this is impossible by Lemma \ref{he-lisse}.

\end{document}